\documentclass{amsart}
\usepackage{amsmath}
\usepackage{amsfonts,mathrsfs}
\newtheorem{theorem}{Theorem}
\theoremstyle{plain}

\newtheorem{conclusion}{Conclusion}

\newtheorem{corollary}{Corollary}

\newtheorem{definition}{Definition}
\newtheorem{example}{Example}

\newtheorem{lemma}{Lemma}

\newtheorem{proposition}{Proposition}
\newtheorem{remark}{Remark}

\numberwithin{equation}{section}

\begin{document}
	
	\title[ $\mathcal{I}^{*}$-sequential topological space ]{ Some properties of $\mathcal{I}^{*}$-sequential topological space }
	
	\author{H. Sabor Behmanush}
	\address{Mersin University, Institute of Sciences Dept. of Mathematics, Mersin, TURKEY}
	\email{}

	\email{h.s.behmanush1989@gmail.com }
	\author{M. K\"u\c c\"ukaslan}
	\address{Mersin University, Faculty of Science Dept. of Mathematics, Mersin, TURKEY}
	\email{}
	\email{mkkaslan@gmail.com and mkucukaslan@mersin.edu.tr}
	\date{}
	\subjclass[2010]{54A20; 54B15; 54C08; 54D55; 26A03; 40A05}

\begin{abstract}
		In this paper, we will define $\mathcal{I}^{*}$-sequential topology on a topological space $(X,\tau)$ where $\mathcal{I}$ is an ideal of the  subset of natural numbers $\mathbb{N}$. Besides the basic properties of the $\mathcal{I}^{*}$-sequential topology,  we proved that   $\mathcal{I}^{*}$-sequential topology is finer than $\mathcal{I}$-sequential topology.
		Further, we will discus main properties of  $\mathcal{I}^{*}$ sequential continuity and  $\mathcal{I}^{*}$ sequential compactness.
\end{abstract}

\maketitle
  Keywords: $\mathcal{I}$- convergence, $\mathcal{I}^{*}$-convergence, Statistical convergence, Sequentially $\mathcal{I}$-topology, Sequentially $\mathcal{I}^{*}$-topology  

\section{Introduction}
The main topic of  this paper is to introduce  $\mathcal{I}^{*}$-sequential topological space as the $\mathcal{I}$-sequential topological space introduced by X. Zhou, L. Liu, S. Lin, in \cite{18}. 

Let us start with a few word about the history of this concept. The concept of statistical convergence was defined by Steinhouse \cite{st} and Fast \cite{4}, independently in 1951. Many studies on statistical convergence have been conducted over the years.
It has many application in different field of mathematics like, summability theory, number theory, trigonometric series, probability theory, measure theory, optimization and approximation and etc.

After that in 2000, authors  P. Kostyrko, T. Salat and W. Wilczynski  introduced the notion of $\mathcal{I}$-convergence in \cite{9} which, in its particular case, coincides with statistical and classical convergence.  Because of the flexibility of the ideal concept, several results relating to $\mathcal{I}$-convergence given in different spaces  \cite{14,15,8,5,11,17,7,13,6}.

Between the years 2012-2019, especially the authors  B.K. Lahiri and P. Das, S.K. Pal, A.K. Banerjee and A. Banerjee, X. Zhou and L. Liv, S. Lin in the papers \cite{10,12,1,18} extend the ideal of $\mathcal{I}$-convergence of a sequence to any topological space and they introduced several properties of this concept in a topological space, recently researches on  $\mathcal{I}$-convergent,  generalized in $\mathcal{I}^{*}$-convergent.

In this paper, we continue the study of topological space which is defined by $\mathcal{I}^{*}$-convergence and we derive some basic properties. Apart from that we are going to compare the  $\mathcal{I}$-sequential topology with $\mathcal{I}^{*}$-sequential topology. 

Recall the notion of  statistical convergence in a topological space. For any subset $A$ in $\mathbb{N}$ the asymptotic density of $A$  is given by $$\delta (A):= \lim _{n\to \infty }\frac{1}{n}\left |\{k \in \mathbb{N}: k\leq n\}  \right |$$
if the limit exists. 

Let $(X, \tau)$ be a topological space, a sequence $\tilde{x}=(x_{n}) \subseteq X$ is said to convergent statistically to a point $x\in X$ if $$\delta( \{ n\in \mathbb{N}, x_{n}\notin U\})=0, $$ holds for any neighborhood $U$ of $x$.

Equivalently,  if $$\delta( \{ n\in \mathbb{N}, x_{n}\in U\})=1,$$
holds. It is denoted by$$ st-\lim _{x \to \infty}x_{n}=x \ or \ x_{n}\overset{st}{\rightarrow} x$$

For a sequence $\tilde{x}=(x_{n})$, let a subset $U$ of $X$, denote by $$A_{U}(\tilde{x}):= \{n\in \mathbb{N} : x_{n} \notin U\}$$ which is also denoted by $A_{U}$. 

It is easy to see in \cite{18} that a sequence in a topological space $X$ converges statistically to a point $x\in X$ if and only if for any neighborhood $U$ of $x$, we have $$\delta (A_{U})=0,$$ (equivalently $\delta ({A}^{c}_{U})=1$).

Now, we recall the concept of sequential topological space, let $(X,\tau)$ be a topological space, a subset $F\subseteq X$ is called sequentially closed if for each sequence $\tilde{x}=(x_{n}) \subset F$ with $x_{n} \to x\in X$ then $x\in F$.

The space $(X,\tau)$ is called sequential topolgical space if each sequentially closed subset of $X$ is closed. 

A subset $U \subseteq X$ is said to be sequentially open if $X-U$ is sequentially closed.

A sequence $\tilde{x}=(x_{n})\subset X$ is said to be eventually in an open subset $U$ of $X$, if there exists $n_{0} \in \mathbb{N}$ such that  $x_{n}\in U$ holds, for all $n>n_{0}$.

Obviously, a subset $U \subseteq X$ is sequentially open if and only if for each sequence $\tilde{x}=(x_{n})$ converging to a point $x \in U$, then $\tilde{x}=(x_{n})$ is eventually in $U$.

\begin{definition}
	Let $S$ be a set then a family $\mathcal{I}\subseteq P(S)$ is called an ideal on $S$ if\\
	(i) For all  $A,B\in \mathcal{I}$ implies $A \cup B \in \mathcal{I}$ \\
	(ii) If $A\in \mathcal{I}$ and $B\subseteq A$ then $B\in \mathcal{I}$.
	
\end{definition}
An ideal on $S$ is called admissible if $\{s\} \in\mathcal{I}$ holds for all $s \in S $; an ideal on $S$ is called proper ideal if $S \notin \mathcal{I}$. 

A proper ideal is called maximal ideal if it is maximal element of the set of all proper ideals on $S$ ordered by inclusion. 

An ideal $\mathcal{I}$ is called non trivial if $\mathcal{I} \neq \phi $ and $S \notin \mathcal{I}$.
\begin{example}
	The families $\mathcal{I}_{Fin}=\{A\subset \mathbb{N}: A \text{ is finite set} \}$  and $\mathcal{I}_{\delta}=\{A\subset \mathbb{N}: \delta(A)=0 \} $ are ideal on the set of natural numbers.
\end{example} 
\begin{remark}
	If we consider;\\ 	
	(i) $\mathcal{I}_{\delta}:=\left\lbrace A\subset \mathbb{N}: \delta (A)=0\right\rbrace,$ then statistical and ideal convergence are coincide.\\ 	
	(ii) $\mathcal{I}_{Fin}:=\left\lbrace A\subset \mathbb{N}: A \text{ is finite set}\right\rbrace,$ then classical convergence and ideal convergence are coincide.
\end{remark}

\begin{example} 
	\label{Exp 1} \cite{8}
	Let $\mathbb{N}=\bigcup_{i=1}^{\infty }\Delta_{i}$ be a decomposition of $\mathbb{N}$ such that for all $i \in \mathbb{N}$ the set $\Delta_{i}$ are infinite subsets of $\mathbb{N}$  and $\Delta_{i}\cap \Delta{j}=\phi$ for all $i\neq j$.
	
	Let $$\mathcal{I}=\{B\subset \mathbb{N}: B \text{ intersect  at most finite number of } \Delta_{j}'s\}. $$
	Then, $\mathcal{I}$ is an admissible and  nontrivial ideal.
\end{example}

The dual notion to the notion of ideal is called filter and defined as follow:
\begin{definition}\cite{14}
	A family $\mathscr{F}\subseteq P(S)$ is said to be filter if \\
	(i)  $A \cap B \in \mathscr{F}$ for all $A, B \in \mathscr{F}$,\\
	(ii)  If $A \in \mathscr{F} \wedge A \subseteq B$, then $B \in \mathscr{F}$.
\end{definition}
A filter $\mathscr{F}$ is called proper if $\phi \notin \mathscr{F}$.
For every non-trivial ideal $\mathcal{I}$ defines a dual filter $$\mathscr{F}(\mathcal{I}):= \{A\subseteq S: S-A \in \mathcal{I}\} $$ on the set $S$ and  we say that $\mathscr{F}(\mathcal{I})$ is the filter associated by $\mathcal{I}$.

In this paper, we are consider $S=\mathbb{N}$ the set of natural numbers and  $(X,\tau)$ denotes a topological space and $\mathcal{I}$ is an admissible ideal defined on the subset of $\mathbb{N}$. Unless otherwise stated, the triple $X$, $\tau$ and $I$ will be displayed in  $(X,\tau, \mathcal{I})$ format.
\begin{definition}\cite{18}
	A sequence $\tilde{x}=(x_{n})$ in a topological space $(X,\tau ,\mathcal{I})$ is said to be $I$-convergent to a point $x \in X,$ if  $$\{n \in \mathbb{N} : x_{n} \notin U\} \in \mathcal{I},$$ holds for any neighborhood $U$ of $x$ and it is denoted by $\mathcal{I}-lim x_{n} =x$ (or $(x_{n})\overset{\mathcal{I}}{\rightarrow}x)$.
\end{definition}

\begin{remark}\label{remark 2}
	If  $\mathcal{I}$ is an admissible ideal, then classical convergence implies $\mathcal{I}$-convergent to $x$.
\end{remark}
Unfortunately, the converse of the Remark \ref{remark 2}  is not true. To see this let $x$ and $y$ be two diffierent elements of $X$ and let $A\in \mathcal{I}$ and consider a sequence $x_n=x$ when $n\in A$ and $x_n=y$ when $n\notin A$. 

Clear that the sequence $(x_n)$ is $\mathcal{I}$ convergent to $y$ but not convergent to $y$ in usual case.
\begin{definition}\cite{ab}
	Let $\mathcal{I}$ be an ideal of $\mathbb{N}$ and $X$ be a topological space, 
	
	(i) For each subset $A\subseteq X$, we define $\bar{A}^{\mathcal{I}}$, the $\mathcal{I}$-closure of $A$ by $$\bar{A}^{\mathcal{I}}=\{x\in X: \exists \ (x_{n})\subset A: x_{n} \to x \}.$$
	
	(ii) A subset $F\subseteq X$ is said to be $\mathcal{I}$-closed if $\bar{F}^{\mathcal{I}}=F$.
	
	(iii) A subset $A\subseteq X$ is said to be $I$-open if $X-A$ is $\mathcal{I}$-closed.
\end{definition}
It is clear that $\bar{\phi}^{\mathcal{I}}=\phi$ and $A\subseteq \bar{A}^{\mathcal{I}}$ hold.  Also,
it can be easily seen that any open subset of topological space $(X,\tau,\mathcal{I})$ is also $\mathcal{I}$-open.

\section{$\mathcal{I}^{*}$- sequential topological space}
Now, we are going to observe the notion of  $\mathcal{I}^{*}$- sequential topological space which is mostly similar to the $\mathcal{I}$- sequential topological space but this topology is finer then $\mathcal{I}$- sequential topological space.
\begin{definition}
Let $(X,\tau,\mathcal{I})$ be a topological space, a sequence $\tilde{x}=(x_{n})$ in $X$ is said to be $\mathcal{I}^{*}$-convergent to a point $x\in X$ if there exists $M\in \mathscr{F}(\mathcal{I})$,$$M=\{m_{1}<m_{2}<\cdots<m_{k}<\cdots\}$$ such that for any neighborhood $U$ of $x$, there exists $ N\in \mathbb{N}$ such that  $x_\textrm{k}\in U$ holds for all  $m_\textrm{k}>N$.
\end{definition} 

If in topological space $(X, \tau,\mathcal{I} )$ the set $X$ is also a field then the definition of  $\mathcal{I}^{*}$-convergence can be reformatted in the form of decomposition theorem as follows: 
\begin{theorem}
Let $(X,\tau, \mathcal{I})$ be a topological space. 
  A sequence $\tilde{x}=(x_{n})\subset X$ is $\mathcal{I}^{*}$-convergent to a point $x\in X$ if and only if it can be written as $x_{n}=t_{n}+s_{n}$ for all $n\in \mathbb{N}$ such that $\tilde{t}=(t_{n})\subset X$ is a $\mathcal{I}_{Fin}$-convergent to $x$ and $\tilde{s}=(s_{n})\subset X$ is non zero only for a set in ideal $\mathcal{I}$.
\end{theorem}
\begin{proof}
   Let  $x_{n}=t_{n}+s_{n}$, for all $n\in \mathbb{N}$ where $t_{n}{\rightarrow} x (\mathcal{I}_{Fin})$ and $(s_{n})$ is non zero only on set from ideal $\mathcal{I}$. Since $t_{n}{\rightarrow} x  (\mathcal{I}_{Fin})$, then for any neighborhood $U$ of $x$  $$\{n\in \mathbb{N}:\ t_{n}\notin U \}\notin F_{in}$$ holds.  Let $$M=\mathbb{N}-\{n\in \mathbb{N}:\ t_{n}\notin U \}.$$
   
   Then,  $s_{n}=0$ for all $n\in M$. So, $x_{n}=t_{n}$ and  this implies that for any neighborhood $U$ of $x$ $x_{n}\in U$ holds for all $n\in M$. Hence, $x_{n}\overset{I^{*}}{\rightarrow} x$.
   
   Conversely, let $x_{n}\overset{I^{*}}{\rightarrow}x$. Then, there exists $M\in F(\mathcal{I})$ and $N\in \mathbb{N}$ such that $x_{n}\in U$ holds for all $n>N$ and $n\in M$  for any neighborhood $U$ of $x$.
   Define the sequence $\tilde{t}=(t_n)$ as $$t_{n}=\left\{\begin{matrix}
x_{n}, & n \in M,\\ 
x, &  n\notin M, 
\end{matrix}\right.$$\\ and the sequence $\tilde{s}=(s_{n})$ as 
$$s_{n}=\left\{\begin{matrix}
0, & n \in M,\\ 
x_{n}-x, &  n\notin M. 
\end{matrix}\right.$$

Then, $t_{n}{\rightarrow}x (\mathcal{I}_{Fin})$ and $(s_{n})$ is nonzero only on a set from ideal $\mathcal{I}$ and
 $x_{n}=t_{n}+s_{n}$ holds  for all $n\in \mathbb{N}$.
\end{proof}

 The fact every $\mathcal{I}^{*}$-convergent sequence is $\mathcal{I}$-convergence sequence was stated by B. K. Lahiri and P. Das in \cite{10}. Because this concept plays an important role in seeing the relation between $\mathcal{I}$-sequential topology and $\mathcal{I}^{*}$-sequential topology so we repeated it again:

\begin{theorem} \cite{10} \label{theorem 1} 
Let $(X,\tau, \mathcal{I})$ be a topological space. If $x_{n}\overset{\mathcal{I}^{*}}{\rightarrow}x$, then  $x_{n}\overset{\mathcal{I}}{\rightarrow}x$.
\end{theorem}
\begin{proof}
Let $x_{n}\overset{\mathcal{I}^{*}}{\rightarrow}x$. Hence, there exists $M\in \mathscr{F}(\mathcal{I})$ such that $M=\mathbb{N} - K$ where $ K \in \mathcal{I}$  $$ M=\lbrace m_{1}<m_{2}<\dots,<m_{k}<\dots\rbrace$$ such that for any neighborhood $U$ of $x$, there exists $N\in \mathbb{N}$ such that  $x_{mk}\in U$ holds for all $m_{k}>N$. This implys that  $$\lbrace n:x_{n}\notin u\rbrace\subset K \cup \lbrace m_{1},m_{2},\dots,m_{N}\rbrace\in \mathcal{I}. $$ 

Hence, $x_{n}\overset{\mathcal{I}}{\rightarrow}x$.
\end{proof}
 Following example will show that the converse of Theorem \ref{theorem 1} is not true:  
 \begin{example}
 Let  $(\mathbb{R}, \tau_{e})$ be Euclidean topology and zero is a limit point.  Let $n \in \mathbb{N}$ and consider  $$B_{n}(0):=(-\frac{1}{2n},\frac{1}{2n}) $$ be a monotonically decreasing open base at zero. 
 
 Define a real valued sequence $\tilde{x}=(x_{n})$  such that $$x_{n}\in B_{n}(0)-B_{n+1}(0)$$ as follows $$(x_{n})=(\frac{2n+1}{4n^{2}+4n}).$$ 
 
 It is clear that $x_{n}\to 0$ when $n\to \infty$. Consider the ideal which is given in Example \ref{Exp 1} and let us note that any $\Delta_{i}$ is a member of $I$ for all $i\in \mathbb{N}$.
 
 Let $\tilde{y}=(y_{n})$ be a sequence defined by $y_{n}=x_{j}$ if $n\in \Delta{j}$. Let $U$ be any open set containing zero. Choose a positive integer $m$ such that $$B_{n}(0)\subset U,$$ for all $n>m$, then $$\{n: y_{n}\notin U\}\subset \Delta_{1}\cup \Delta_{2} \cup \Delta_{3}...\cup \Delta_{m} \in \mathcal{I}$$ this implies that $y_{n}\overset{\mathcal{I}}{\rightarrow}0$.
 
 Now suppose that $y_{n}\overset{\mathcal{I}^{*}}{\rightarrow}0$, then there exists $H\in \mathcal{I}$ such that for $$M=\mathbb{N}-H=\{m_{1}<m_{2}<...<m_{k}...\}\in \mathscr{F}(\mathcal{I})$$ we have: there exists $N\in \mathbb{N}$ such that $x_{m_{k}}\in U$  for all $m_{k}>N$ and for any neighborhood $U$ of zero. Let $l\in \mathbb{N}$ and assume that $$H\subset \Delta_{1}\cup \Delta_{2} \cup \Delta_{3}...\cup \Delta_{l}$$ then  $\Delta_{i}\subset \mathbb{N}-H$ holds for all for $i>l+1$.  
 
 Therefore, for each $i>l+1$, there is infinitely many $k$'s such that $y_{m_{k}}=x_{i}$. But, then the $lim {y_{n_{k}}}$ doesn't exists because $x_{i}\neq x_{j}$ for all $i\neq j$. 
 \end{example}
    \begin{theorem} \label{theorem 2}
    Let $\mathcal{I}$ and $\mathcal{J}$ be two ideals of $\mathbb{N}$ such that $\mathcal{I}\subseteq \mathcal{J}$ and $\tilde{x}=(x_{n})$ be a sequence in a topological space $(X,\tau)$. If $x_{n}\overset{\mathcal{I}^{*}}{\rightarrow}x$, then $x_{n}\overset{\mathcal{J}^{*}}{\rightarrow}x$.
    \end{theorem}
    \begin{proof}
     Let $(x_{n})\overset{\mathcal{I}^{*}}{\rightarrow}x$. Then, there exists $M\in \mathscr{F}(\mathcal{I})$ as $$M=\lbrace m_{1}<m_{2}<\dots,<m_{k}<\dots\rbrace$$ such that for any neighborhood $U$ of $x$, there exists $N\in \Bbb{N}$ such that $x_{m_{k}}\in U$,  for all $m_{k}>N$ holds. 
     
     Hence, $M\in \mathscr{F}(\mathcal{I})$ implies that  $\Bbb{N}-M\in \mathcal{I} $. From the assumption of Theorem we have $\Bbb{N}-M\in \mathcal{J} $ so $  M\in \mathscr{F}(\mathcal{J})$ hence $(x_{n})\overset{J^{*}}{\rightarrow}x$.
     \end{proof}
    \begin{remark}
    	Converse of Theorem \ref{theorem 2} is not true, in general.
    \end{remark}
 \begin{definition}
Let  $(X,\tau,\mathcal{I})$ be a topological space.

(i) For each subset $A$ of $X$ we define $\overline{A}^{\mathcal{I}^{*}}$  ($\mathcal{I}^{*}$-Closure of $A$) by
    $$\overline{A}^{\mathcal{I}^{*}} = \lbrace x\in X: \exists (x_{n}) 
      \subset A \ such\ that\ x_{n}\overset{I^{*}}{\rightarrow}x  \rbrace$$
      
(ii) A subset $F \subseteq X$ is said to be $\mathcal{I}^{*}$-closed if $ \overline{F}^{\mathcal{I}^{*}}= F$.

(iii) A subset $U\subseteq X$ is said to be $\mathcal{I}^{*}$-open if $X-U$ is $\mathcal{I}^{*}$-closed.
\end{definition} 

\begin{remark}
It is clear that $\overline{\phi}^{\mathcal{I}^{*}}$= $\phi$ and $A\subset \overline{A}^{\mathcal{I}^{*}}$ are true for all $A \subseteq X$.
\end{remark}

 \begin{theorem}
  Let $(X,\tau)$ be a topological space, and $\mathcal{I}$ be finite ideal. Then, $\mathcal{I}$-convergent of a sequence and $\mathcal{I}^{*}$-convergence of a sequence are coincide.
\end{theorem}
\begin{proof}
We proved that if $x_{n}\overset{\mathcal{I}^{*}}{\rightarrow}x$ then  $x_{n}\overset{\mathcal{I}}{\rightarrow}x$, for any ideal of $\mathbb{N}$.

Let a sequence $x_{n}\overset {\mathcal{I} }\rightarrow x$, then for any neighborhood $U$ of $x$ , $$A:=\lbrace n \in \Bbb{N}: x_{n} \notin U \rbrace  \in \mathcal{I}.$$ 

Consider $M= \Bbb{N}-A \in \mathscr{F}(\mathcal{I})$ and arrange $M$ as $$M=\lbrace m_{1}<m_{2}<\cdots<m_{k}<\cdots\rbrace.$$ 

Since the set $A$ is finite, then there exists $N \in \Bbb{N}$ such that $x_{m_{k}} \in U$ holds for all $m>N$. Therefore, this implies  $x_{n}\overset{\mathcal{I}^{*}}{\rightarrow}x$.  
\end{proof}

\begin{theorem}
Let $(X,\tau, \mathcal{I})$ be a topological space. If $\mathcal{I}$ is admissible ideal, then every $\mathcal{I}$-open subset of $X$ is $\mathcal{I}^{*}$- open subset of $X$.
\end{theorem} 
 \begin{proof} Let $U$ be an $\mathcal{I}$-open subset of $X$.  Then, $X-U$ is $\mathcal{I}$-closed, that is $X-U=\overline{X-U}^\mathcal{I}$ holds. To prove  $$X-U=\overline{X-U}^{\mathcal{I}^{*}}$$ it is sufficient to show that 
 $$\overline{X-U}^{\mathcal{I}^{*}} \subset X-U$$ holds.
 
 Let $x \in\overline{X-U}^{\mathcal{I}^{*}}$ be an arbitrary point. Then, there exists a sequence $\tilde{x}=(x_{n})\subset X-U$ such that $x_{n}\overset{\mathcal{I}^{*}}{\rightarrow}x$ holds, then by Theorem 1 we have  $x_{n}\overset{\mathcal{I}}{\rightarrow}x$. This implies that $$x \in \overline{X-U}^{\mathcal{I}}=X-U. $$ 
 
 Hence, the proof of Theorem is completed.
 \end{proof}

 \begin{corollary}
Let $(X,\tau, \mathcal{I})$ be a topological space. If $\mathcal{I}$ is a finite ideal then $A\subset X$ is $\mathcal{I}$-open if and only if $A$ is $\mathcal{I}^{*}$-open subset of $X$.
\end{corollary} 
 \begin{definition}
  Let $A$ be a subset of topological space $(X,\tau, \mathcal{I})$. We define $A^{o^{\mathcal{I}^{*}}}$(called $\mathcal{I}^{*}$ interior of $A$) as  $$A^{o^{\mathcal{I}^{*}}}:= A-\overline{(X-A)}^{\mathcal{I}^{*}}.$$
\end{definition}
 
\begin{proposition}
 Let $A$ be a subset of topological space $(X,\tau, \mathcal{I})$. Then, the set $A$ is $\mathcal{I}^{*}$-open if and only if $A^{o^{\mathcal{I}^{*}}}=A$.
 \end{proposition}
\begin{proof}
 Let $A$ be $\mathcal{I}^{*}$-open subset of topological space $(X,\tau, \mathcal{I})$. Then, the related set $X-A$ is $\mathcal{I}^{*}$-closed, that is $$X-A=\overline{X-A}^{\mathcal{I}^{*}}.$$ 
 
 Therefore, 
 $$A^{o^{\mathcal{I}^{*}}}= A-\overline{(X-A)}^{\mathcal{I}^{*}}=A-(X-A)=A$$
 holds.
 
Conversely assume that  $A=A^{o^{\mathcal{I}^{*}}}$ holds. 
As we defined $$A^{o^{\mathcal{I}^{*}}}:= A-\overline{(X-A)}^{\mathcal{I}^{*}}.$$ Then,
 $$A= A-\overline{(X-A)}^{\mathcal{I}^{*}}$$ implies that  $A \cap \overline{(X-A)}^{\mathcal{I}^{*}}=\phi$.   Therefore, $\overline{(X-A)}^{\mathcal{I}^{*}}\subset X-A$  and this show that  ${(X-A)}^{\mathcal{I}^{*}}= X-A$. 
 
 Hence, $X-A$ is $\mathcal{I}^{*}$- closed and $A$ is $\mathcal{I}^{*}$-open.
\end{proof}

\begin{theorem} Let $A$ be a subset of topological space $(X,\tau,\mathcal{I})$. Then, the following assertions are equivalent:

(i) $A$ is $\mathcal{I}^{*}$- closed.

(ii) $A=\bigcap\lbrace F:F \text{ is  } \mathcal{I}^{*}-\text{ closed and } A \subset F \rbrace$.

\end{theorem}
\begin{proof}
$(i) \Rightarrow (ii)$ is obvious. Let $$A= \bigcap \lbrace F: F  \ is \ \mathcal{I}^{*}- \text{ closed and } A \subset F \rbrace. $$ 

To show that $\bar A^{\mathcal{I}^{*}} = A$ hol,ds it is sufficient to prove that $$\bar A^{\mathcal{I}^{*}} \subseteq A$$ holds.  Let $x_{0}\in\bar A^{\mathcal{I}^{*}}$ is an arbitrary point, then there exists a sequence $(x_{n})\subset A$ such that $(x_{n})$ is $\mathcal{I}^{*}$-convergent to $x_{0}$. Assume that $x_{0}\notin A$. So, $$x_{0}\notin \bigcap \{ F:F \ is\ \mathcal{I}^{*}-\text{ closed and } A \subset F \rbrace.$$ 

Then, there exists an $\mathcal{I}^{*}$-closed subset $F$ of $X$ which $x_{0}\notin F$. Since $(x_{n})\subset A$ and $ A \subset F$ then $x_{0}$ must be in $F$ which is contradiction to our assumption.
\end{proof}

\begin{theorem}
 Let $A$ be a subset of topological space $(X,\tau,\mathcal{I})$. Then, the following  assertions are equivalent: 
 
 (i) $A$ is $\mathcal{I}^{*}$- open.
 
 (ii) $A=\bigcup\lbrace U: U \ is\ \mathcal{I}^{*} -\text{open and } U\subset A\rbrace$.

\end{theorem}
\begin{proof}
$(i) \Rightarrow (ii)$ is obvious.

 Conversely,  let $$A=\bigcup\lbrace U:U \ is\ {\mathcal{I}}^{*}-\text{open and }U\subset A \}.$$ 
 
 To prove  $A$ is $\mathcal{I}^{*}$- open subset of $X$, we must to show that  $A= A^{o^{\mathcal{I}^{*}}}$ holds. It is known  that $A^{o^{\mathcal{I}^{*}}}$ always subset of $A$. So, it is sufficient to show that $A\subset A^{o^{\mathcal{I}^{*}}}$. 
 
 Let $x_{0}\in A$ be an arbitrary point, then there is an open subset $U$ of $A$ such that $x_{0}\in U$. Since $U\subset A$ then $x_{0}\in A^{o^{\mathcal{I}^{*}}}$ and this implies that  $$A\subset A^{o^{\mathcal{I}^{*}}}$$ holds. So, proof is ended.
 \end{proof} 
  
\begin{definition}
 Let $\mathcal{I}$  be an ideal and $U$ be a subset of topological space $X$. A sequence $\tilde{x}=(x_{n})\subset X$ is $\mathcal{I}^{*}$-eventually in $U$ if there exists $M\in \mathscr{F}(\mathcal{I})$ such that $x_{m} \in U$ holds for all $m\in M.$
 \end{definition}
 
\begin{proposition}
 Let $\mathcal{I}$ be a maximal ideal of $\Bbb{N}$ and $(X,\tau)$ be a topological space. Then, a subset $U\subseteq X$ is $\mathcal{I}^{*}$-open if and only if each $\mathcal{I}^{*}$-convergent sequence to a point $x\in U$ in $X$ is $\mathcal{I}^{*}$-eventually in $U$. 
 \end{proposition}
\begin{proof}
Let us assume  each $\mathcal{I}^{*}$-convergent sequence to a point $x_{0}\in U$ is  $\mathcal{I}^{*}$-eventually in $U$. That is,  if $(x_{n})\subset X$ is a sequence which $(x_{n})$ is $\mathcal{I}^{*}$-convergent to $x_{0} \in U$, then there exists $M \in \mathscr{F}(\mathcal{I})$ such that $x_{m} \in U$ holds  for all $m\in M$.  Now, we are going to show that $U$ is  $I^{*}$-open. For this purpose, we will show that $$\overline{X-U}^{\mathcal{I}^{*}}= (X-U).$$ 

It is clear that $(X-U) \subset \overline{X-U}^{\mathcal{I}^{*}}$. So, to finish the proof it is sufficient to show that $$\overline{X-U}^{\mathcal{I}^{*}} \subseteq (X-U)$$ holds. 

Let $x\in \overline{X-U}^{I^{*}}$ be an arbitrary point. Then, there exists a sequence $(x_{n})\subset (X-U)$ such that  $(x_{n})$ is $ \mathcal{I}^{*}$-convergent to $x$. We must to show that $x\in X-U$. Assume that $x\notin X-U$. Then, $ x \in U$. Since every $\mathcal{I}^{*}$ convergent sequence to a point of $U$ is eventually on $U$, then there exists $M \in \mathscr{F}(\mathcal{I})$ such that  $x_{m}\in U$ for all $m\in M$, but we have $x_{n}\in X-U$, for all $n$ which is contradiction. Hence $x\in X-U$ and $U$ is $\mathcal{I}^{*}$-open.

Conversely, let $U$ is $\mathcal{I}^{*}$ open subset of $(X, \tau)$. Let $(x_{n})\subset X$ be an be an $\mathcal{I}^{*}$-convergent sequence which $\mathcal{I}^{*}$-converges to a point $x \in U$. Since $U$ is $\mathcal{I}^{*}$-open subset of the space $X$, then it is neighborhood of the point $x$, also as the sequence $x_{n}\overset{\mathcal{I}^{*}}{\rightarrow}x$ then $x_{n}\rightarrow x$ so  
 $$E:= \lbrace n: x_{n}\notin U \rbrace \in \mathcal{I}$$ 
 
 Therefore, 
 $$\Bbb{N}-E=\lbrace n: x_{n}\in U \rbrace  =M \in \mathscr{F}(\mathcal{I}) $$

Hence, there exists $m \in M$ such that $x_{m}\in U$. This implys that  $(x_{n})$ is $\mathcal{I}^{*}$-eventually in $U$ and this completed the proof.
\end{proof} 

\begin{theorem}
Let $\mathcal{I}$ be an admissible ideal of $\Bbb{N}$ and $(X,\tau)$ be a topological space, if $U$ and $V$ are  $\mathcal{I}^{*}$-open  subsets of $X$, then $U\cap V$ is  $\mathcal{I}^{*}$-open subset of $X$.
\end{theorem}
\begin{proof}
 
Let $(x_{n})$ be an $\mathcal{I}^{*}$-convergent sequence in $X$ which convergent to a point $x \in U \cap V $, then $x \in U$ and $x \in V$. Since  $U$ and $V$ are $\mathcal{I}^{*}$-open sets and the sequence $(x_{n})$ is $\mathcal{I}^{*}$-converging to a point $x$ in $U$ also in $V$. 

So, by Proposition 3 the sequence $(x_{n})$ is $\mathcal{I}^{*}$-eventually in $U$ and also in $V$, then there exists $M_{1}\in \mathscr{F}(\mathcal{I})$  and  $M_{2}\in \mathscr{F}(\mathcal{I})$ such that  $x_{m} \in U$ for all $m \in M_{1}$ and $x_{m}\in V$ for all $m \in M_{2}$. Let $ M= M_{1}\cap M_{2}$, then there exists $M\in \mathscr{F}(\mathcal{I})$ such that $x_{m}\in U\cap V$ holds for all $m\in M$. This shows that $U\cap V$ is $\mathcal{I}^{*}$-open subset of $X$. 
\end{proof} 
\begin{theorem}
Let $\mathcal{I}$ be a maximal ideal of $\Bbb{N}$ and $(X,\tau)$ be a topological space. A sequence $(x_{n})\subset X$ is  $\mathcal{I}^{*}$-convergent to an element $x \in X $ if and only if  for any  $\mathcal{I}^{*}$-open subset $U$ of $X$ with $x \in U$, there is $M\in \mathscr{F}(\mathcal{I})$ such that  $x_{m}\in U$, for all $m \in M$.
\end{theorem}

\begin{proof}
Let $\mathcal{I}$ be a maximal ideal and $(x_{n})$ be an $\mathcal{I}^{*}$-convergent sequence which $\mathcal{I}^{*}$-converges to $x\in X$. Let $U$ be an $\mathcal{I}^{*}$-open subset of $X$ with $x \in U$. Then $(x_{n})$ will be  $\mathcal{I}^{*}$-eventually in $U$. Hence there exists a set $M \in \mathscr{F}(\mathcal{I})$ such that $ x_{m}\in U$, for all $ m \in M $.

The converse of the Theorem is clear by Definition of $I^{*}$-convergence. So, it is ommited here.
\end{proof}

\begin{theorem}
Let  $(X,\tau, \mathcal{I})$ be a topological space. Then, the family $$\tau _{\mathcal{I}^{*}} := \lbrace U \in P(X): U \ is \ \mathcal{I}^{*}-\text{open  set}  \rbrace$$ is a topology on $X$. (It is called $\mathcal{I}^{*}$-sequential topology on $X$)
\end{theorem}
\begin{proof}
 It is obvious that $X$ and $\phi$ are $\mathcal{I}^{*}$-open sets, because the set $X$ and empty set $\phi$ are $\mathcal{I}^{*}$-closed so their complement are $\mathcal{I}^{*}$-open subset of topological space $X$.
 
Let $ U \text{ and }V \in \tau_{\mathcal{I}^{*}}$. By Theorem 8, we can say that  $U \cap U \in \tau_{\mathcal{I}^{*}}$. Hence, finite intersection of $\mathcal{I}^{*}$-open sets is $\mathcal{I}^{*}$-open.

Let $(U_{\alpha })_{\alpha \in \Lambda }$ be an arbitrary family of elements of $\tau_{\mathcal{I}^{*}}$. We are going to show that $$\bigcup_{\alpha \in \Lambda }U_{\alpha} \in \tau_{\mathcal{I}^{*}}.$$ 

Since $$X-\bigcup_{\alpha \in \Lambda }U_{\alpha} = \bigcap_{\alpha \in \Lambda }(X-U_{\alpha}),$$ then it is sufficient to show that $$\bigcap_{\alpha \in \Lambda }(X-U_{\alpha})$$ is $\mathcal{I}^{*}$- closed. That is, $$\overline{\bigcap_{\alpha \in \Lambda }(X-U_{\alpha})}^{\mathcal{I}^{*}}=\bigcap_{\alpha \in \Lambda }(X-U_{\alpha}).$$ 

Let $x \in \overline{\bigcap_{\alpha \in \Lambda }(X-U_{\alpha})}^{\mathcal{I}^{*}}.$ Then, there exists a sequence $$(x_{n}) \subset \bigcap_{\alpha \in \Lambda }(X-U_{\alpha})$$ such that $x_{n}\overset{\mathcal{I}^{*}}{\rightarrow}x$ satisfies. Therefore, for all $\alpha \in \Lambda$ the sequence $(x_{n}) \subseteq (X-U_{\alpha})$ with $x_{n}\overset{\mathcal{I}^{*}}{\rightarrow}x$. Since $X-U_{\alpha}$ is $\mathcal{I}^{*}$- closed for all $\alpha \in \Lambda$, then $x\in X-U_{\alpha}$. 

Hence, $x \in \bigcap_{\alpha \in \Lambda }(X-U_{\alpha}) $
thus $\bigcap_{\alpha \in \Lambda }X-U_{\alpha}$ is $\mathcal{I}^{*}$- closed.
\end{proof}
\begin{theorem}
   If $\mathcal{I}$ is admissible ideal and the topological space $(X,\tau)$ has no limit point, then every $\mathcal{I}^{*}$-open set is $\mathcal{I}$-open set.  
\end{theorem}
\begin{proof}
Let $U$ be an $\mathcal{I}^{*}$-open set. To prove $U$ is $I$-open set, it is enouh to show that $X-U$ is $\mathcal{I}$-closed. That is,  we are going to focus that  $X-U=\overline{X-U}^{\mathcal{I}}$. It is clear that  $X-U \subseteq \overline{X-U}^{\mathcal{I}}$. 

The proof will finish if we show that  $\overline{X-U}^{\mathcal{I}} \subseteq X-U$. Let  $x \in \overline{X-U}^{\mathcal{I}}$. Then, there exists a sequence $(x_{n})\subset X-U$ such that $(x_{n})$ is $\mathcal{I}$-converges to $x$. Since $\mathcal{I}$ is admissible and $X$ has no limit point, then by \cite{10} the sequence $(x_{n})$ will $\mathcal{I}^{*}$-converge to $x$. 

Therefore, $x \in \overline{X-U}^{\mathcal{I}^{*}}$. As $U$ is $\mathcal{I}^{*}$- open, then $\overline{X-U}^{\mathcal{I}^{*}}=X-U$ holds. This implies that $x \in X-U$ which complete the proof of Theorem.
\end{proof}

\begin{corollary}
 If $X$ has no limit point and $\mathcal{I}$ is admissible ideal, then the topological spaces $(X,\tau_{\mathcal{I}})$ and $(X,\tau _{\mathcal{I}^{*}})$ are coincide.
\end{corollary}

\begin{definition}
Let  $\mathcal{I}$ be an ideal of  $\mathbb{N}$
\end{definition}

\begin{definition}
Let  $(X,\tau_{1},\mathcal{I})$ and $(Y,\tau_{2},\mathcal{I})$ be two topological space and $f:X \rightarrow Y$ be a mapping. 
The function $f$ is called;

(i) $\mathcal{I}^{*}$-continuous if $U$ is $\mathcal{I}^{*}$-open subset of $Y$ then  $f^{-1}(U)$ is $\mathcal{I}^{*}$-open subset of  $X$.
  
(ii) Sequentially $\mathcal{I}^{*}$-continuous if for each sequence $(x_{n})$ in $X$ which $(x_{n})$ is $\mathcal{I}^{*}$ convergent to $x$, we have $f(x_{n})$ is $\mathcal{I}^{*}$-convergent to $f(x)$.

\end{definition}

\begin{theorem}
Let $(X,\tau_{1}, \mathcal{I})$ and $(Y, \tau_{2},\mathcal{I})$ be two topological space and $f:X\rightarrow Y$ be a function. Then, $f$ is sequentially $\mathcal{I}^{*}$-continuous if and only if $f$ is $\mathcal{I}^{*}$-continuous function.
\end{theorem}
\begin{proof}
Let $f$ be sequentially $\mathcal{I}^{*}$-continuous function then for any sequence $\tilde{x}=(x_{n}) \subset X$, if  $(x_{n})$,  $\mathcal{I}^{*}$-converges to $x$ then $f(x_{n}),\mathcal{I}^{*}$-converges to $f(x)$. Let $U$ be any $\mathcal{I}^{*}$-open set in $Y$ assume that $ f^{-1}(U)$ is not $\mathcal{I}^{*}$- open set in $X$ so $X - f^{-1}(U)$ is not $\mathcal{I}^{*}$- closed; i.e, $$X - f^{-1}(U) \neq \overline{X-f^{-1}(U)}^{\mathcal{I}^{*}}.$$ 

This implies that $\overline{X-f^{-1}(U)}^{\mathcal{I}^{*}}$ is not subset of $X - f^{-1}(U)$ so there exists a point $x \in \overline{X-f^{-1}(U)}^{\mathcal{I}^{*}}$ such that $x \notin X - f^{-1}(U)$, this means that there exists a sequence $(x_{n}) \subset X - f^{-1}(U)$ such that it is $\mathcal{I}^*$- converging to $x$ and $x \in f^{-1}(U)$. 

Since $f$ is sequentially continuous so $f(x_{n})$,  $I^{*}$-converging to $f(x)$ this employs that $f(x_{n}) \subset Y-U$ which is not in case so $f^{-1}(U)$ is $\mathcal{I}^{*}$-open subset of $X$.
 
Conversely: Let $f:X \to Y$ be an $\mathcal{I}^{*}$-continuous mapping, let $x_{n}\overset{\mathcal{I}^{*}}{\rightarrow} x$, then for any neighborhood $U$ of $x$, there exists $N \in \mathbb{N}$ and $M \in \mathscr{F}(\mathcal{I})$ such that $x_{m_{k}} \in U$ for all $m_{k}\in M$. Let $V$ be any $\mathcal{I}^{*}$-open neighborhood of $f(x)$, then $f^{-1}(V)$ is $\mathcal{I}^{*}$-open in $X$ and contain $x$, then there exists $N\in \mathbb{N}$, $M\in \mathscr{F}(\mathcal{I})$ such that $x_{m_{k}}\in f^{-1}(V)$ this implies that $f(x_{m_{k}}) \in V$ hence $f(x_{n})\overset{\mathcal{I}^{*}}{\rightarrow}f(x)$.    
\end{proof}

%\begin{definition}
%Let $I$ be an ideal of $\Bbb{N}$ and $(X, \tau)$ be a topological space
%$X$ is called $I^{*}$-separated if for any two distinct point $x,y \in X$ there exists two $I^{*}$-open subsets $U$ and $V$ of $X$ such that $x \in U$ and $y \in V$ and $U \cap V \neq \phi$.
%\end{definition}

\section{Sequentially  $I^{*}$-compact subsets}

The notion of compactness is one of the most significant topological properties. Compactness was formally introduced by M. Frechet  in 1906. After that many types of compactness introduced by mathematicians. 

The concept of $I$-compactness was deﬁned by Newcomb \cite{newcomb} and had been studied by Rancin \cite{rancin}. 

In this section, we are going to observe the notion of $\mathcal{I}^{*}$-sequentially compact which is a generalization of $\mathcal{I}$-sequentially compact. 
\begin{definition}
Let  $(X,\tau,\mathcal{I})$ be a topological space.

(i) $X$ is called  $\mathcal{I}^{*}$-separated if for any two distinct points $x,y \in X$ there is  $\mathcal{I}^{*}$-open subsets $U$ and $V$ containing $x$ and $y$ respectively, and $U \cap V = \phi$.
    
(ii) A subset $F$ in $X$ is called  $\mathcal{I}^{*}$-compact if for every   $\mathcal{I}^{*}$-open cover of $F$ there exists a finite subcover of $F$.
    
(iii)  A subset $F \subset X$ is said to sequentially $\mathcal{I}^{*}$-compact if any sequence $(x_{n})\subset F$ has an $\mathcal{I}^{*}$-convergent sub sequence $(x_{n_{k}})\subset F$ such that  $x_{n_{k}}\overset{I^{*}}{\rightarrow}x\in F$. 
   
(iv) The space $X$ is said to be countably $\mathcal{I}^{*}$-compact if for every countably $\mathcal{I}^{*}$-open cover of $X$ there exists a finite subcover.
\end{definition}

\begin{theorem}\label{theorem 13}
Let $(X,\tau,\mathcal{I})$ be a topological space and $(Y,\tau_{Y},\mathcal{I})$ be a subspace of $X$. If the set $Y$ is $\mathcal{I}$-compact, then it is $\mathcal{I}^{*}$-compact.
\begin{proof}
Let $Y$ be a $\mathcal{I}$-compact subset of $X$. So, every $\mathcal{I}$-open cover of $Y$ has a finite subcover. Since every $I$-open set is $\mathcal{I}^{*}$-set so every $\mathcal{I}^{*}$-open cover of $Y$ has finite subcover, this show that $Y$ is $\mathcal{I}^{*}$-compact.
\end{proof}
\end{theorem}
\begin{definition} \cite{15}
  Let $X$ be a normed space and $\mathcal{I}$ be an ideal of $\mathbb{N}$. A sequence $\tilde{x}=(x_{n})$ in $X$ is $\mathcal{I}$-bounded if there exist $N>0$ such that $$\lbrace n \in \Bbb{N}:\lVert x_{n} \rVert >N \rbrace \in \mathcal{I}$$ holds.
  
\end{definition}
\begin{definition}

Let $X$ be a normed space, and $\mathcal{I}$ be an ideal of $\mathbb{N}$. A sequence $\tilde{x}=(x_{n})$ in $X$ is said to be $\mathcal{I}^{*}$-bounded if there exists $M\in \mathscr{F}(\mathcal{I})$ such that $(x_{n})_{n \in M}$ is bounded.
\end{definition}
\begin{theorem}
Let $X$ be a normed space and $\mathcal{I}$ be an ideal of $\mathbb{N}$. Then, every $\mathcal{I}$-bonded sequence is $\mathcal{I}^{*}$-bounded.
\end{theorem}
\begin{proof}
Assume that $(x_{n})\subset X$ is  $\mathcal{I}$-bounded sequence in $X$. Then, there exists $K>0$ such that  $$\lbrace n:\lVert x_{n} \rVert >K \rbrace \in \mathcal{I} $$ hold. If we denote $$M:=\{m: \lVert x_{n} \rVert <K\}$$ Then, $M\in \mathscr{F}(\mathcal{I})$ and $\lVert x_{n} \rVert <K$, for all $n \in M$. Hence,  $(x_{n})$ is $\mathcal{I}^{*}$- bounded sequence.
\end{proof}

\begin{corollary}
Let $X$ be a normed space and $\mathcal{I}$ be an ideal  of $\mathbb{N}$. Then, every bounded sequence is $\mathcal{I}^{*}$-bounded.
\end{corollary}
\begin{proof}
Let $X$ be a normed space and $(x_{n})\subset X$ be a bounded sequence in $X$. Then, the sequence $(x_{n})$ is $\mathcal{I}$-bounded \cite{ab} and by Theorem \ref{theorem 13} it is $\mathcal{I}^{*}$-bounded.
\end{proof}

\begin{theorem}
Let $(X,\tau, \mathcal{I})$ be a topological space and $\Bbb{R}$ under it's usual topology and $f:X \rightarrow\Bbb{R}$ a sequentially $\mathcal{I}^{*}$-continuous mapping. If $A$ is sequentially $\mathcal{I}^{*}$-compact subset of $X$, then $f(A)$ is $\mathcal{I}^{*}$-bounded. 
\begin{proof}
Let $f(A)$ is not $\mathcal{I}^{*}$-bounded then there exists a sequence $(y_{n})$ in $f(A)$ such that it is not $\mathcal{I}^{*}$-bounded then $$\lbrace n \in \Bbb{N}:\left|  y_{n} \right|  <M \rbrace \notin \mathscr{F}(\mathcal{I}).$$ also there exists a sequence $(x_{n})\subset A$ such that  $f(x_{n})=y_{n}$ holds $n \in \Bbb{N}$. 

Since $A$ is sequentially $\mathcal{I}^{*}$- compact, then there exists a convergent subsequence $(x_{n_{k}})$ of $(x_{n})$ which is $\mathcal{I}^{*}$-convergent  to a point $x_{0}$ of $A$. Because $f$ is sequentially $\mathcal{I}^{*}$-continuous function then $f(x_{n})$ is $ \mathcal{I}^{*}$-convergent to $f(x_{0})$. So, there exists $E\in \mathscr{F}(\mathcal{I}), (\Bbb{N}-E\in \mathcal{I})$ $$E=\lbrace m_{1}< m_{2}< \cdots < m_{k}< \cdots \rbrace$$ such that for any neighborhood $U$ of $f(x_{0})$, there exists $N \in \Bbb{N}$ such that        $f(x_{n_{m_{k}}}) \in U$ holds for all $m_{k}>N$. This implies that  $(y_{n})=f(x_{n_{k}})$ is $\mathcal{I}$-convergent to $f(x_{0})$. Then, for any neighborhood $U$ of $f(x_{0})$,$$\lbrace n \in \Bbb{N}:\lvert f(x_{n}) \rvert >M \rbrace \in \mathcal{I}$$so  $$\lbrace n \in \Bbb{N}:\lvert x_{n} \rvert <M \rbrace \in \mathscr{F}(\mathcal{I}).$$ Which is not in case so $f(A)$ is $\mathcal{I}^{*}$-bounded.
\end{proof}
\end{theorem}

\begin{lemma}
Let $X$ and $Y$ be topological spaces. $X$ and $Y$ are sequentially $\mathcal{I}^*$-compact, if and only if $X \times Y$ is. 
\end{lemma}

\begin{proof}
Let $X$ and $Y$ be $\mathcal{I}^*$- compact then for any sequence $(X_{n}) \subset X$ and $(y_{n}) \subset Y$ there exists convergent sub sequence $(x_{n_{k}}) \subset X$ and $(y_{n_{j}})\subset Y$. Which are $\mathcal{I}^*$*- converging to a point $x \in X$ and $y \in Y$ respectively. Thus, for sequence $(x_{n},y_{n})\subset X \times Y$ there exists sub sequence $(x_{n_{k}},y_{n_{j}})$, which are $\mathcal{I}^{*}$-converging to  $(x,y)\in X \times Y$. Hence $X\times Y$ is sequentially $\mathcal{I}^{*}$-compact.

Conversely: Let $P_{x}: X \times Y \rightarrow X$ and $P_{y}: X \times Y \rightarrow Y$ be two projections and $X \times Y$ be compact, then  $P_{x}(X \times Y)$ and $P_{y}(X\times Y)$ are sequentially $\mathcal{I}^{*}$- compact, this show that  $X$  and $Y$ are sequentially $\mathcal{I}^{*}$- compact.
\end{proof}

\begin{theorem}
Let $X$ and $Y$ be topological spaces. If $X$ is sequentially $\mathcal{I}^{*}$-compact and $f:X\rightarrow Y$ is sequentially $\mathcal{I}^{*}$-mapping, then $f(X)$ is sequentially $\mathcal{I}^{*}$-compact.
\end{theorem}
\begin{theorem}
Let $\mathcal{I}$ and $\mathcal{J}$ be two ideal of $\Bbb{N}$ such that $\mathcal{I}\subset \mathcal{J}$ and $X$ be a topological space. If $U\subset X$ is $\mathcal{J}^{*}$- open then it is $\mathcal{I}^{*}$-open.
\end{theorem}
\begin{proof}
Let $U$ be $\mathcal{J}^{*}$-open then $X-U$ is $\mathcal{J}^{*}$-closed and $X-U= \overline {X-U}^{\mathcal{J}^{*}}$ holds. 

We must to prove $\overline{X-U}^{\mathcal{I}^{*}}\subset X-U$. Let $x \in \overline{X-U}^{\mathcal{I}^{*}}$ be an arbitrary point, then there exists a sequence $(x_{n})\subset X-U$ such that $(x_{n})$ is $\mathcal{I}^{*}$-convergent to a point $x \in X-U$, then by Theorem \ref{theorem 2} the sequence  $(x_{n})$, $\mathcal{J}^{*}$ converges to $x$. Hence, $x \in \overline{X-U}^{\mathcal{J}}=X-U$ this implies that  $x \in X-U$ and $U$ is $\mathcal{J}^{*}$-open.
\end{proof}

 \begin{corollary}
 Let $\mathcal{I}$ and $\mathcal{J}$ be two ideal of $\mathbb{N}$ such that $\mathcal{I}\subset \mathcal{J}$. Then, the topological space $\tau _\mathcal{J}^{*}$ is finer then the topological space $\tau _\mathcal{I}^{*}$.
 
 \end{corollary}  
 \begin{theorem} \label{17}
     Let $\mathcal{I}$ be an ideal and $X$ be topological space, if every sub-sequence $(x_{n_{k}})$ of $(x_{n}) \subseteq X$ is $\mathcal{I}^{*}$- convergent to a point $x_{0} \in X$ then $(x_{n})$ is $\mathcal{I}^{*}$- convergent to $x_{0}$.
 \end{theorem}
\begin{proof}
    Let $(x_{n)}$ is not $\mathcal{I}^{*}$-convergent to point $x_{0}$ then for all $M \in \mathscr{F}(\mathcal{I})$ and for all $N \in \mathbb{N}$ there exists $n_{k} > N$ such that $x_{n_{k}} \notin U$, where $U$ is any neighborhood of $x_{0}$, if we take $N=1$ then there exists the sub-sequence $(x_{n_{k}}) \notin U$, for all $n_{k} > 1$ this means that there exists a sub-sequence of $(x_{n})$ which is not converging to the point of $x_{0}$ which is contradiction.
\end{proof}
The converse of  Theorem \ref{17} is not true as follow:
\begin{example}
    Let consider $(\mathbb {R}, \tau_{e})$, the set of real numbers with it's usual topology, let $\mathcal{I}$ be any ideal and $K \in \mathscr{F}(\mathcal{I})$ define a sequence as 
    $$y_{n}=\left\{\begin{matrix}
2^n & n \notin K\\ 
\frac{1}{n} &  n\in K 
\end{matrix}\right.$$\\
$(y_{n}) $ is $\mathcal{I}^{*}$-convergent to zero but the subsequence $y_{n_{k}} = 2^{n} , n\notin K$ is not $\mathcal{I}^{*}$-convergent. 
\end{example}
\begin{conclusion}
	
In the paper, we defined the $\mathcal{I}^{*}$-sequential topology on a topological space $(X,\tau)$ and we proved that $\mathcal{I}^{*}$-sequential topology is finer then $\mathcal{I}$-sequential topology. That is; $\tau_{\mathcal{I}}\prec \tau_{\mathcal{I}^{*}}$. 

We also observed that under the condition of if the space $X$ has no limit point and $\mathcal{I}$ be an admissible ideal then, $\mathcal{I}$-sequentially topology and $\mathcal{I}^{*}$-sequentially topology are coincide, i.e $\tau_{\mathcal{I}}= \tau_{\mathcal{I}^{*}}$. 

Also,  Lemma 2 in the paper \cite{ab} stated that "Every subsequence of an $\mathcal{I}$-convergent sequence in a topological space is also $\mathcal{I}$-convergent" but in Example 3 of this paper we saw that it is not true in $\mathcal{I}^{*}$-sequentially topological space.

As a continuation of this study, some questions can be asked: 

$Q1:$ Is there a finer topology than $\mathcal{I}^{*}$-sequentially topology space?

$Q2:$  Is there any topological  between $\mathcal{I}$-sequential topological space and $\mathcal{I}^{*}$-sequential topological space?

\end{conclusion}

\end{document}